\documentclass{amsart}

\usepackage[utf8]{inputenc} 

\usepackage{amsmath, amssymb, amsthm, amsfonts}
\usepackage[american]{babel}
\usepackage{hyperref}
\usepackage{color}
\usepackage{tikz-cd}
\usepackage{mathtools}
\usepackage{comment}
\usepackage{tikz}
\usepackage{rotating}
\usepackage{MnSymbol}
\usepackage{csquotes}

\usetikzlibrary{arrows}

\newtheorem{theorem}{Theorem}

\newtheorem{definition}[theorem]{Definition}

\numberwithin{equation}{section}

\newcommand{\N}{\mathbb{N}}
\newcommand{\Z}{\mathbb{Z}}
\newcommand{\Q}{\mathbb{Q}}
\newcommand{\R}{\mathbb{R}}
\newcommand{\C}{\mathbb{C}}

\newcommand{\LL}{\mathcal{L}}
\newcommand{\F}{\mathcal{F}}
\newcommand{\AP}{\operatorname{AP}}
\newcommand{\QP}{\operatorname{QP}}

\title{A Strichartz estimate for quasiperiodic functions}
\author{Friedrich Klaus}

\begin{document}

\begin{abstract}
    In this work we prove a Strichartz estimate for the Schrödinger equation in the quasiperiodic setting. We also show a lower bound on the number of resonant frequency interactions in this situation.
\end{abstract}

\maketitle

\section{Introduction}

A major open problem in the theory of nonlinear Schrödinger equations is the existence, uniqueness and large time behavior of solutions with prescribed initial data which does not decay as $|x| \to \infty$. This problem is open even for the one-dimensional, cubic nonlinear Schrödinger equation
\begin{equation}\label{eq:nls}
    iu_t + u_{xx} = 2|u|^2 u,
\end{equation}
which is completely integrable. There is one prominent example of initial data which is particularly easy to write down, but also particularly tenacious, namely
\begin{equation}\label{eq:initialdata}
    u_0(x) = \cos(x) + \cos(\sqrt{2}x).
\end{equation}
This function is neither periodic (since $\sqrt{2}$ is irrational), nor does it have a fixed asymptotic profile at $\pm \infty$. Thus it cannot be treated within the standard theory of periodic or decaying Sobolev spaces, and it does not fit either into the Gross-Pitaevskii theory or the theory of the tooth problem for NLS (see \cite{klaus1} for the latter).

One may naively think of trying to solve \eqref{eq:nls} in $L^\infty$ or $C^0$. This fails though, because the topologies of these spaces are not well enough adapted to dispersive PDE. Indeed, already for the linear Schrödinger equation strong oscillations can cause dispersive blow-up in $L^\infty$ \cite{bonaponce}. One way to prevent these issues is to assume boundedness of the derivatives of the initial data, and this leads to local wellposedness either in Sobolev spaces $W^{k,\infty}$ \cite{dss} or modulation spaces $M_{\infty,1}$ \cite{klaus2}. The example \eqref{eq:initialdata} fits well into these settings. 

There is also another more direct way to treat \eqref{eq:initialdata} locally in time, which was layed out in \cite{oh1} and which we briefly sketch. Note first that the function $u_0$ in \eqref{eq:initialdata} is almost periodic in the sense of Bohr:

\begin{definition}
    A function $f: \R \to \C$ is almost periodic if $f$ is continuous and if for every $\varepsilon > 0$ there exists $L > 0$ such that every interval of length $L$ contains a number $\tau$ such that
    \[
        \sup_{x\in\R} |f(x-\tau)-f(x)| < \varepsilon.
    \]
    We define $\AP(\R)$ to be the space of almost periodic functions.
\end{definition}
Recent results by Chapouto--Killip--Visan \cite{ckv} on KdV with initial data in  $\AP(\R)$ suggest that this space is not well adapted to dispersive problems either, mainly since its definition assumes continuity of the function. Indeed, they show an instant loss of continuity for an explicit class of initial data. Thus we have to restrict to a subspace of $\AP(\R)$ as follows: 

Given an almost periodic function $u \in \AP(\R)$ we can define the Fourier coefficients of the function as
\[
    \F u(k) = \hat u(k) = \lim_{L \to \infty} \frac{1}{2L}\int_{-L}^L u(x)e^{-ikx}\, dx.
\]
These Fourier coefficients are defined for all $k \in \R$, but only countably many of them are non-zero (see \cite[Section 1]{oh1} for more details). We call $\sigma(u) = \{k \in \R: \hat u(k)\neq 0\}$ the frequency set of $u$. Now for $\omega \in \R^\N$ linearly independent over $\Q$ consider those functions with $\sigma(u) \subset \omega$ and with $\ell^1$ summable Fourier coefficients. It turns out that for all $\omega$ this space, $A_\omega$, is a Banach algebra, hence local wellposedness of \eqref{eq:nls} can be proven quite easily in it (see \cite[Theorem 1.7]{oh1}).

All of these results in $W^{k,\infty}, M_{\infty,1}$ and $A_\omega$ are local in time. In fact, global in time wellposedness of \eqref{eq:nls} with initial data \eqref{eq:initialdata} is an open question. The problem is that even though there are infinitely many conserved quantities for \eqref{eq:nls}, these are not adapted to the aforementioned spaces. 

To be more precise we take a closer look at the Fourier support of $u_0$. Clearly, $\sigma(u_0) \subset \Z + \sqrt{2}\Z$, and the same holds for all of its powers and products with complex conjugates of itself. This makes $u_0$ a quasiperiodic function:
\begin{definition}
    $u \in \AP(\R)$ is called quasiperiodic if there exist $N$ real numbers $\omega_1, \dots, \omega_N$ which are linearly independent over $\Q$ such that $\sigma(u) \subset \omega_1\Z + \dots + \omega_N \Z$. We define $\QP(\omega_1,\dots,\omega_N)$ to be the space of those quasiperiodic functions.
\end{definition}
We introduce the notation $\Omega = \Z + \sqrt{2}\Z$. Consider the $\LL^2$-norm defined by
\begin{equation}\label{eq:L2}
    \|u\|_{\LL^2} = \Big(\sum_{k \in \Omega} |\hat u(k)|^2\Big)^\frac{1}2.
\end{equation}
Then, as can be checked by hand by a formal calculation, a solution $u$ of \eqref{eq:nls} preserves the $\LL^2$-norm. The problem is though that this norm is not strong enough to bound the norms of any of the aforementioned spaces. It is thus very natural to define $\LL^2$ as the closure of functions in $\QP(1,\sqrt{2})$ with $\ell^1$-summable Fourier coefficients, with respect to the $\LL^2$ norm\footnote{This is closely related to the Besicovitch almost periodic functions $B^2$.}, and ask for local (and thus global) wellposedness of \eqref{eq:nls} in this space. This is an open problem though.

This short work addresses this problem, both in the positive and the negative direction. We begin with the positive direction. Both in the non-periodic and the periodic wellposedness theory for NLS equations, Strichartz estimates play a crucial role. These estimates connect a mixed space-time $L^p$ norm of the solution of the linear Schrödinger equation to the $L^2$ norm of the initial data. Motivated by the Plancherel type identity
\[
    \|u\|_{\LL^2} = \Big(\lim_{L \to \infty} \frac{1}{2L}\int_{-L}^L |u(x)|^2\, dx\Big)^{\frac12},
\]
we define the equivalent of $L^p$ spaces in the almost periodic setting as the closure of $\QP(1,\sqrt{2})$ functions with $\ell^1$-summable Fourier coefficients, with respect to the norm\footnote{Alternatively we can just say that $u \in \LL^p$ if and only if $u$ is measurable and $|u|^{\frac{p}{2}}\in \LL^2$.}
\begin{equation}\label{eq:Lp}
    \|u\|_{\LL^p} = \Big(\lim_{L \to \infty} \frac{1}{2L}\int_{-L}^L |u(x)|^p\, dx\Big)^{\frac1p}.
\end{equation}

Our main result is the following Strichartz estimate in the quasiperiodic case:
\begin{theorem}\label{lem:strichartz}
    Let $g \in \LL^2$ and define the Schrödinger group $S(t)$ by
    \begin{equation}
        \F\big(S(t)g\big)(k) = e^{-itk^2}\hat g(k).
    \end{equation}
    Then the following Strichartz estimate holds:
    \begin{equation}\label{eq:strichartz}
        \|S(t)g\|_{\LL^4_{t,x}} \lesssim \|g\|_{\LL^2},
    \end{equation}
\end{theorem}
Theorem \ref{lem:strichartz} will be proven in Section \ref{sec2}. To the best of the author's knowledge this is the first Strichartz estimate in an almost periodic setting. Its proof is not difficult and highly influenced by the proof for the periodic Strichartz estimate originally given in \cite{bourgain} and outlined in \cite[Section 2]{erdogantzirakis}.

Clearly the hope would be to use Theorem \ref{lem:strichartz} to prove wellposedness of \eqref{eq:nls} in $\LL^2$ by a $TT^*$ argument. The usual $TT^*$ argument building on \eqref{eq:strichartz} gives an estimate of the form
\[
    \lim_{T \to \infty} \frac1{2T}\Big\|\int_{-T}^T S(t-s)F(s,x)\, ds\Big\|_{\LL^4_{t,x}} \lesssim \|F\|_{\LL^{4/3}_{t,x}},
\]
which can be rewritten as
\[
    \lim_{T,T',R \to \infty} \strokedint_{-T'}^{T'}\strokedint_{-R}^R\Big|\strokedint_{-T}^T S(t-s)F(s,x)\, ds\Big|^4 \, dxdt \lesssim \|F\|_{\LL^{4/3}_{t,x}}^4.
\]
Unfortunately this \emph{does not} imply the estimate
\[
    \Big\|\int_{-T}^T S(t-s)F(s,x)\, ds\Big\|_{\LL^4_{t,x}} \lesssim T\|F\|_{\LL^{4/3}_{t,x}}
\]
uniformly in $F$, even if $T$ is large enough, because the rate of convergence in the limit may depend on the function $F$. Such an estimate would have been enough to solve the fixed-point equation for \eqref{eq:nls}.

There is also the negative direction of thinking though. We introduce the notation $k = k_x + \sqrt{2}k_y \in \Omega$. In some way, $\Omega$ is two-dimensional over $\Z$ (more precisely it is a free module over $\Z$ of rank two). In comparison, the periodic NLS in two dimensions is mass-critical, and an $L^4$ estimate does not hold, see \cite{kishimoto}! Thus maybe an illposedness result in $\LL^2$ can be achieved instead?

Such an illposedness result is usually (and also in \cite{kishimoto}) shown by proving a lower bound on the first nonlinear term in the Picard iteration,
\begin{equation}\label{eq:g}
    g(t,x) = \int_0^t S(t-s) (|S(s)u_0|^2 S(s))u_0 \, ds.
\end{equation}
On the Fourier side, this term satisfies
\begin{align*}
    |\hat g(t,k)| = \sum_{k_1-k_2+k_3 = k}\int_0^t e^{it'\Phi} \hat u_0(k_1)\overline{\hat u_0(k_2)}\hat u_0(k_3) dt',
\end{align*}
where
\[
    \Phi = 2(k_1-k)(k_3-k).
\]
In the periodic case, those frequency interactions for which $\Phi = 0$ are usually called resonant. In the situation where the frequencies are allowed to take values in $\Omega$, there are additional interactions which make $\Phi$ arbitrarily small but non-zero. We thus define the resonant set
\[
    \Gamma(k,N) = \{k_1,k_2,k_3 \in \Omega, k_1-k_2+k_3=k, |k_{i,x}|\leq N, |k_{i,y}|\leq N, |\Phi| \leq 1\}.
\]
Our second result is the following:
\begin{theorem}\label{lem:resonant}
    For $|k_x|, |k_y| \leq N/2$ the resonant set satisfies
    \begin{equation}\label{eq:cardinality}
        |\Gamma(k,N)|\gtrsim N^2 \log N.
    \end{equation}
\end{theorem}
Theorem \ref{lem:resonant} is also proven in Section \ref{sec2}. It is not clear to the author whether or not Theorem \ref{lem:resonant} is enough to conclude illposedness. He tried by considering similar initial data as in \cite{kishimoto} but did not succeed. It seems to him that the usually applicable principle of the resonant part being dominant in \eqref{eq:g} is not that easy to prove. Possibly a nonlinear smoothing estimate similar to what is proven in \cite[Section 4]{ckv} could help in this situation.

\subsubsection*{Acknowledgements} Funded by the Deutsche Forschungsgemeinschaft (DFG, German Research Foundation) – Project-
ID 258734477 – SFB 1173.

\section{Proofs of Theorems 1 and 2}\label{sec2}

\begin{proof}[Proof of Theorem \ref{lem:strichartz}]
    Since $\|f\|_{\LL^4_{t,x}}^4 = \|f^2\|_{\LL^2}^2$, Plancherel shows
    \begin{align*}
        \|S(t)g\|_{\LL^4_{x}}^4 &= \sum_{k \in \Omega} \Big|\F_x\big((S(t)g)^2\big)(k)\Big|^2\\
        &= \sum_{k \in \Omega} \Big|\sum_{l_1 + l_2 = k} e^{-itl_1^2}\hat g(l_1)e^{-itl_2^2}\hat g(l_2)\Big|^2\\
        &= \sum_{k \in \Omega} \sum_{l_1 + l_2 = j_1 + j_2 = k} e^{-it(l_1^2 + l_2^2 - j_1^2 - j_2^2)}\hat g(l_1)\hat g(l_2)\overline{\hat g(j_1)\hat g(j_2)}.
    \end{align*}
    Note that
    \[
        \lim_{L \to \infty} \frac1{2L} \int_{-L}^L e^{i\xi x}\, dx = \begin{cases} 1, \quad \text{if} \quad \xi = 0,\\
        0, \quad \text{else}.\end{cases}
    \]
    Thus, by integrating in time we cancel all non-resonant terms. Defining
    \[
        A_{p,q} = \{(k_1,k_2) \in \Omega^2: k_1 + k_2 = p, k_1^2 + k_2^2 = q\},
    \]
    we can rewrite
    \begin{align*}
        \|S(t)g\|_{\LL^4_{t,x}}^4 &= \lim_{L\to\infty} \frac1{2L}\int_{-L}^L \|S(t)g\|_{\LL^4_{x}}^4 dt\\
        &= \sum_{p, q \in \Omega} \sum_{(l_1, l_2), (j_1,j_2) \in A_{p,q}} \hat g(l_1)\hat g(l_2)\overline{\hat g(j_1)\hat g(j_2)}\\
        &= \sum_{p, q \in \Omega} \Big|\sum_{(l_1, l_2)\in A_{p,q}} \hat g(l_1)\hat g(l_2)\Big|^2.
    \end{align*}
    We are lead to bound the number of elements in $A_{p,q}$ and we claim that this set is finite independently of $p,q$, more precisely
    \begin{equation}\label{eq:bound4}
        |A_{p,q}|\leq 4.
    \end{equation}
    If this claim holds we can estimate with Cauchy-Schwarz,
    \begin{align*}
        \|S(t)g\|_{\LL^4_{t,x}}^4 &\leq 4\sum_{p, q \in \Omega} \sum_{(l_1, l_2)\in A_{p,q}} |\hat g(l_1)|^2|\hat g(l_2)|^2\\
        &= 4\|g\|_{\LL^2}^4,
    \end{align*}
    proving Lemma \ref{lem:strichartz}.
    
    To prove \eqref{eq:bound4} we write $k_2 = p- k_1$ and investigate the solutions $k_1 \in \Omega$ of
    \begin{align*}
        q_x + \sqrt2 q_y = q &= k_1^2 + (p-k_1)^2 \\
        &= k_{1,x}^2 + (p_x - k_{1,x})^2 + 2(k_{1,y}^2 + (p_y - k_{1,y})^2) \\
        &\qquad + 2\sqrt{2}\big(k_{1,x}k_{1,y} + (p_x - k_{1,x})(p_y-k_{1,y})\big).
    \end{align*}
    This reduces to the system of equations
    \begin{align}
        \label{eq:quadratic}q_x &= k_{1,x}^2 + (p_x - k_{1,x})^2 + 2(k_{1,y}^2 + (p_y - k_{1,y})^2)\\
        \label{eq:linear}q_y &= 2k_{1,x}k_{1,y} + 2(p_x - k_{1,x})(p_y-k_{1,y})
    \end{align}
    We rewrite \eqref{eq:quadratic} by noticing
    \[
        2(k_{1,x}^2 + (p_x - k_{1,x})^2) = 4k_{1,x}^2 - 4p_x k_{1,x} + 2p_x^2 = (2k_{1,x} - p_x)^2 + p_x^2,
    \]
    and similarly in the $y$-coordinate, as
    \[
        (2k_{1,x} - p_x)^2 + 2(2k_{1,y} - p_y)^2 = 2q_x - p_x^2 - 2p_y^2.
    \]
    By defining $X = 2k_{1,x}-p_x$, $Y = 2k_{1,y} - p_y$, and $A = 2q_x - p_x^2 - 2 p_y^2$, this equation reads
    \begin{equation}\label{eq:quadratic2}
        X^2 + 2Y^2 = A.
    \end{equation}
    We turn to \eqref{eq:linear}. Here we write
    \begin{align*}
        2k_{1,x}k_{1,y} + 2(p_x - k_{1,x})(p_y-k_{1,y}) &= 4k_{1,x}k_{1,y} + 2p_xp_y - 2(p_xk_{1,y}+p_yk_{1,x})\\
        &= (p_x - 2k_{1,x})(p_y-2k_{1,y}) + p_xp_y,
    \end{align*}
    which, by defining $B = q_y - p_xp_y$, transforms \eqref{eq:linear} into
    \begin{equation}\label{eq:linear2}
        XY = B.
    \end{equation}
    We plug \eqref{eq:linear2} into \eqref{eq:quadratic2} to obtain the quartic equation
    \[
        X^4 - AX^2 + 2B^2 = 0,
    \]
    which has at most four solutions, proving \eqref{eq:bound4}.
\end{proof}

\begin{proof}[Proof of Theorem \ref{lem:resonant}]
    To prove the inequality \eqref{eq:cardinality} define
    \[
        k_{1,x} - k_x = a, k_{3,x}-k_x = b, k_{1,y}-k_y = c, k_{3,y}-k_y = d,
    \]
    and write $\Z_N = \Z\cap\{|k|\leq \N\}$. Since
    \[
    \begin{split}
        \Phi &= 2((k_{1,x}-k_x)(k_{3,x}-k_x) + 2(k_{1,y}-k_y)(k_{3,y}-k_y) \\
        &\qquad + \sqrt{2}((k_{1,x}-k_x)(k_{3,y}-k_y) + (k_{3,x}-k_x)(k_{1,y}-k_y)))
    \end{split}
    \]
    this shows that it is enough to find a lower bound on the cardinality of
    \[
        \tilde \Gamma = \{a,b,c,d \in \Z_N, |ab+2cd + \sqrt{2}(ad+bc)|\leq 1\},
    \]
    and we claim that
    \[
        |\tilde \Gamma| \gtrsim N^2 \log N.
    \]
    Indeed, write $ab + 2cd = p, ad + bc = q$. The number of $p,q \in \Z_{N^2}$ with
    \[
        |p + \sqrt{2}q| \leq 1
    \]
    is approximately $N^2$ since the latter condition defines a strip of width $1$ and length approximately $N^2$. We rewrite the transformation $ab + 2cd = p, ad + bc = q$ as
    \[
        \left(\begin{matrix} a& 2c\\ c& a\end{matrix}\right)\left(\begin{matrix}b\\d\end{matrix}\right) = \left( \begin{matrix}p\\q\end{matrix}\right).
    \]
    For each pair $(p,q)$ we find exactly one integer solution $(b,d)$ if
    \begin{equation}\label{eq:Pell}
        a^2 - 2c^2 = 1,
    \end{equation}
    because in this case the inverse matrix has integer coefficients. Thus to prove our claim it is enough to show that there are approximately $\log N$ integer solutions of \eqref{eq:Pell}.
    
    The Diophantine equation \eqref{eq:Pell}  is known as Pell's equation. A special solution of \eqref{eq:Pell} is given by
    \[
        a_1 = 3, c_1 = 2.
    \]
    Furthermore, we can construct infinitely many more solutions of the equation by setting
    \[
        \left( \begin{matrix}a_{n+1}\\c_{n+1}\end{matrix}\right) = \left(\begin{matrix} 3& 4\\ 2& 3\end{matrix}\right)\left(\begin{matrix}a_{n}\\c_n\end{matrix}\right).
    \]
    In particular, we can construct approximately $\log N$ many integer solutions of \eqref{eq:Pell} $a,c \in \Z_{N/2}$, which proves our claim.
\end{proof}

\newpage

\bibliographystyle{plain}

\end{document}